\documentclass[a4paper]{amsart}
\usepackage[utf8]{inputenc}
\usepackage{amssymb,todonotes,xfrac,url}
\theoremstyle{plain}
\newtheorem{thm}{Theorem}
\newtheorem{cor}[thm]{Corollary}
\newtheorem{lem}[thm]{Lemma}
\newtheorem{obs}[thm]{Observation}
\theoremstyle{remark}

\newcommand{\st}{\;\mid\;}
\newcommand{\legendre}[2]{(\tfrac{#1}{#2})}
\newcommand{\divides}{\mid}
\newcommand{\ndivides}{\nmid}
\newcommand{\CL}{\mathcal{L}}
\newcommand{\FF}{\mathbb{F}}

\newcommand{\og}{\overline{g}}
\newcommand{\OA}{\overline{A}}
\newcommand{\OB}{\overline{B}}
\newcommand{\OG}{\overline{G}}
\DeclareMathOperator{\rank}{rank}
\def\clap#1{\hbox to 0pt{\hss#1\hss}}

\def\mathclap{\mathpalette\mathclapinternal}

\def\mathclapinternal#1#2{\clap{$\mathsurround=0pt#1{#2}$}}
\title[Yet another proof of the quadratic reciprocity law]{Yet another proof\\of the quadratic reciprocity law}
\author[A. Czoga{\l}a \and P. Koprowski]{Alfred Czoga\l a \and Przemys{\l}aw Koprowski}
\address{Institute of Mathematics\\
University of Silesia\\ Bankowa 14\\ 40-007 Katowice,
Poland} \email{alfred.czogala@us.edu.pl}
\address{Institute of Mathematics\\
University of Silesia\\ Bankowa 14\\ 40-007 Katowice,
Poland} \email{przemyslaw.koprowski@us.edu.pl}
\begin{document}
\begin{abstract}
We present a new proof of the celebrated quadratic reciprocity law. Our proof is based on group theory.
\end{abstract}
\maketitle
Among all mathematical results it is the quadratic reciprocity law which possibly has the highest number of published proofs. The web page \begin{center}\url{http://www.rzuser.uni-heidelberg.de/~hb3/fchrono.html}\end{center} lists a total of 246 (at the time of writing) distinct proofs. In this paper we present yet another proof. The idea of our proof is to some extent inspired by Rousseau's proof \cite{Rousseau91}. 

Our proof of the quadratic reciprocity law is based on some basic facts from group theory. If $(G,+)$ is a finite abelian group, then the quotient group $\sfrac{G}{2G}$ can be treated as a linear space over~$\FF_2$. Recall that $2$-rank of~$G$, denoted $\rank_2(G)$, is the dimension of this vector space. Equivalently, since every finite abelian group is a direct sum of cyclic groups, the $2$-rank of~$G$ is the number of cyclic summands of even orders. Denote by~$G_2$ the subgroup of~$G$ consisting of all elements of orders not exceeding~$2$:
\[
G_2 := \bigl\{ g\in G\st 2g = 0\bigr\}.
\]
Then~$G_2$ is an elementary $2$-group isomorphic to $\sfrac{G}{2G}$. It follows that the $2$-rank of~$G$ is the dimension of~$G_2$ treated as a $\FF_2$-linear space. In particular~$G_2$ is isomorphic to $\FF_2^{\rank_2(G)}$ and we have:

\begin{obs}\label{obs_rank2_as_log2}
With the above notation, $\rank_2 G = \log_2|G_2|$.
\end{obs}

\begin{lem}\label{lem_rank_a}
Let $(G,+)$ be a finite abelian group and $a := \sum_{g\in G} g$ be the sum of all the elements in~$G$.
\begin{itemize}
\item If $\rank_2(G) \neq 1$, then $a = 0$.
\item If $\rank_2(G) = 1$, then the element~$a$ has order~$2$ in~$G$.
\end{itemize}
\end{lem}

\begin{proof}
As above let~$G_2$ be the subgroup of~$G$ consisting of elements of orders $\leq 2$. For every element $g\in G\setminus G_2$, we have $g \neq -g$. Combining such elements into pairs $(g, -g)$ we obtain
\[
a = \sum_{g\in G} g = \sum_{g\in G_2} g.
\]
In particular, if~$G$ has an odd number of elements, then $\rank_2(G) = 0$ and $a = 0$, as claimed. 

Now assume that $|G|$ is even and denote $m := \rank_2(G)$. Recall that the linear space~$G_2$ and~$\FF_2^m$ are isomorphic. If $m = 1$, then $\sum_{v\in \FF_2} v = 1$ and so $a = \sum_{g\in G_2}g$ is the unique element of~$G$ of order~$2$. On the other hand, if $m > 1$ then for every $i \leq m$ in the vectors space~$\FF_2^m$ there are precisely $2^{m-1}$ vectors~$v$ such that the $i$-th coordinate of~$v$ equals~$1$. It follows that the $i$-th coordinate of the sum $\sum_{v\in \FF_2^m}v$ equals $2^{m-1}\cdot 1 = 0$. This means that this sum is the null vector. Using our isomorphism $G_2 \cong \FF_2^m$ we see that $a = \sum_{g\in G_2}g = 0$, as desired.
\end{proof}

\begin{lem}
Let $C_1, \dotsc, C_k$ be cyclic groups of even orders and $c_1\in C_1, \dotsc, c_k\in C_k$ be the unique elements of order~$2$ in each group. Let $G := C_1\times \dotsb \times C_k$ and take a subgroup $\Gamma := \bigl\{ (0,\dotsc, 0), (c_1, \dotsc, c_k)\bigr\}$. Finally set $\OG := \sfrac{G}{\Gamma}$. Then
\[
\rank_2(\OG) =
\begin{cases}
k, &\text{if $4$ divides $|C_i|$ for every $i\leq k$,}\\
k-1, &\text{if there is $i\leq k$ such that $|C_i|\equiv 2\pmod{4}$.}
\end{cases}
\]
\end{lem}

\begin{proof}
Notice that if $\og\in \OG$ is an element of a form $\og = g + \Gamma$, then $\og\in \OG_2$ if an only if $2g\in \Gamma$. We first show that the $2$-rank of~$\OG$ is always greater or equal $k-1$. To this end consider the product:
\[
A := \bigl( C_1 \bigr)_2\times \dotsb \times \bigl( C_k \bigr)_2 = \{0,c_1\}\times \dotsb\times \{0,c_k\}.
\]
It is clear that~$A$ is a subgroup of~$G$ and~$\Gamma$ is a subgroup of~$A$. Identify the quotient group $\OA := \sfrac{A}{\Gamma}$ with the image of~$A$ in~$\OG$. The group~$A$ has precisely~$2^k$ elements, hence~$\OA$ has $2^{k-1}$ elements. Moreover $\OA\subset \OG_2$ and so we derive the lower bound for the $2$-rank: $\rank_2(\OG) = \log_2(\OG_2)\geq k-1$, by Observation~\ref{obs_rank2_as_log2}.

Now, assume that there is at least one cyclic summands of~$G$ whose order is not divisible by~$4$, say $|C_j|\equiv 2\pmod{4}$. Take an element $\og = g + \Gamma\in \OG_2$, where $g = (g_1, \dotsc, g_k)$. We have $2g\in \Gamma$ so that $2g_j = 0$ or $2g_j = c_j$. The latter is impossible since~$|C_j|$ is not divisible by~$4$. It follows that $g_j \in \{ 0, c_j\}$ and so $2g = (0,\dotsc, 0)$. This means that $g\in A$ and this implies that $\OG_2 = \OA$. Consequently $|\OG_2| = 2^{k-1}$, hence $\rank_2\OG = k-1$.

Finally, assume that~$4$ divides the orders of all the cyclic groups $C_1, \dotsc, C_k$. For every $i \leq k$ there are exactly two elements of order~$4$ in~$C_i$. Fix one of then and denote it~$d_i$. The other one is then $d_i + c_i$.  Consider a ``translation'' of the group~$A$:
\[
B := (d_1, \dotsc, d_k) + A = \bigl\{ (g_1, \dotsc, g_k)\st g_i\text{ has order }4\}
\]
and let~$\OB$ be the image of~$B$ in~$\OG$. We then have $|B| = |A| = 2^k$ and $|\OB| = |\OA| = 2^{k-1}$. Moreover $\OA\cap \OB = \emptyset$ and $\OA\cup \OB \subset \OG_2$. Consequently~$\OG_2$ has at least $2^k$ elements and so $\rank_2(\OG)\geq k$. On the other hand, the $2$-rank of the quotient group~$\OG$ cannot exceed the $2$-rank of~$G$ itself, and the latter equals~$k$, the number of cyclic summands. All in all, we obtain $\rank_2(\OG) = k$, as claimed.
\end{proof}

From now on let $p,q$ be two distinct (but fixed) prime numbers. Denote by $G$ the direct product $\FF_p^\times\times \FF_q^\times$ of invertibles modulo~$p$ and~$q$. Consider a subgroup $\Gamma := \bigl\{ (1,1), (-1,-1)\bigr\}$ of~$G$ and set $\OG := \sfrac{G}{\Gamma}$. Apply the above lemma to~$\OG$.

\begin{cor}\label{cor_pq_rank}
With the above notation:
\begin{itemize}
\item If $p\equiv q\equiv 1\pmod{4}$, then $\rank_2\OG > 1$.
\item If either $p\equiv 3\pmod{4}$ or $q\equiv 3\pmod{4}$, then $\rank_2\OG = 1$.
\end{itemize}
\end{cor}

Borrowing an idea from \cite{Rousseau91}, we consider a set~$\CL$ of representatives of all cosets of~$\Gamma$ in~$G$. Let:
\[
\CL := \Bigl\{ (k\bmod p, k\bmod q) : 0 < k < \tfrac{pq}{2}, p\ndivides k, q\ndivides k \Bigr\}
\]
The following fact was proved in \cite{Rousseau91}. We reprove it here only to make this paper completely self-contained.

\begin{lem}\label{lem_prod_L}
The product of all elements of~$\CL$ equals
\[
\Bigl( (-1)^{\frac{q-1}{2}}\cdot \legendre{q}{p},\ (-1)^{\frac{p-1}{2}}\cdot \legendre{p}{q} \Bigr).
\]
\end{lem}

\begin{proof}
The definition of~$\CL$ is symmetric in~$p,q$ and so is the assertion of the lemma. Hence it suffices to prove that the equality holds just for one coordinate. For instance, take a look at the first coordinate:
\begin{multline*}
\prod_{\mathclap{(k,k)\in \CL}} k
= \frac{\phantom{\quad}\prod\limits_{\mathclap{\substack{k < \sfrac{pq}{2}\\ p\ndivides k}}} k\phantom{\quad}}{\prod\limits_{\mathclap{\substack{k < \sfrac{pq}{2}\\ q\divides k}}} k}\\
= \frac{\biggl(\prod\limits_{0 < k < p}k\biggr)\cdot \biggl(\prod\limits_{0 < k < p}(p+k)\biggr)\dotsm \biggl(\prod\limits_{0 < k < p}\bigl(\frac{q-3}{2}\cdot p + k\bigr)\biggr)}{(q)(2q)\dotsm (\frac{p-1}{2}\cdot q)}\cdot \prod_{0<k<\frac{p}{2}} \Bigl(\frac{q-1}{2}\cdot p + k\Bigr)
\end{multline*}
Each product in the numerator evaluates to $(p-1)!$ hence equals $-1$ by Wilson's theorem. Analogously, the last product evaluates to the factorial $(\frac{p-1}{2})!$. Finally, the denominator equals 
\[
(q)(2q)\dotsm (\tfrac{p-1}{2}\cdot q) = q^{\frac{p-1}{2}}\cdot \bigl(\tfrac{p-1}{2}\bigr)! = \legendre{q}{p}\cdot \bigl(\tfrac{p-1}{2}\bigr)!
\]
by Euler's criterion. All in all, the whole formula simplifies to $(-1)^{\frac{q-1}{2}}\cdot \legendre{q}{p}$.
\end{proof}

We are now ready to present the new proof of the quadratic reciprocity law.

\begin{proof}[Proof of the quadratic reciprocity law]
Consider possible remainders of $p,q$ modulo~$4$. First assume that $p\equiv q\equiv 1\pmod{4}$, then $\rank_2\OG > 1$ by Corollary~\ref{cor_pq_rank} and so Lemma~\ref{lem_rank_a} implies that the product of all elements from~$\CL$ lies in~$\Gamma$. Thus, both coordinates are either simultaneously equal~$1$ or simultaneously equal  $-1$. In particular the first coordinate is the same as the second one, hence we have
\[
(-1)^{\frac{q-1}{2}}\legendre{q}{p} = (-1)^{\frac{p-1}{2}}\legendre{p}{q}  
\]
by Lemma~\ref{lem_prod_L}. This shows that $\legendre{q}{p} = \legendre{p}{q}$.

Conversely assume that at least one of the two primes is congruent to~$3$ modulo~$4$. Therefore $\rank_2\OG = 1$ and so by Lemma~\ref{lem_rank_a} the product of elements from~$\CL$ has order~$2$ in the quotient group~$\OG$. Thus the product equals $(1,-1)\cdot \Gamma = (-1,1)\cdot \Gamma$. In particular the two coordinates are opposite to each other:
\[
(-1)^{\frac{q-1}{2}}\legendre{q}{p} = -(-1)^{\frac{p-1}{2}}\legendre{p}{q} .
\]
Now, if $p\not\equiv q\pmod{4}$, then $(-1)^{\frac{q-1}{2}} = - (-1)^{\frac{p-1}{2}}$ and again we have $\legendre{q}{p} = \legendre{p}{q}$. On the other hand, if $p\equiv q\equiv 3\pmod{4}$, then $(-1)^{\frac{q-1}{2}} = (-1)^{\frac{p-1}{2}}$ and so $\legendre{q}{p} = -\legendre{p}{q}$. This concludes the proof.
\end{proof}


\end{document}